\documentclass[12pt, a4paper]{amsart}

\usepackage{amsmath,amssymb,amsthm, url}
\usepackage[all]{xy}
\usepackage{enumerate}
\usepackage{fullpage}
\usepackage[dvipdfmx]{graphicx}
\usepackage{wrapfig}
\usepackage{mathtools}
\usepackage{ascmac}
\usepackage[foot]{amsaddr}
\usepackage{mathrsfs}
\usepackage[mathscr]{euscript}
\usepackage{bbm}
\usepackage{setspace}
\usepackage{hyperref}

\theoremstyle{plain}
\numberwithin{equation}{section}
\newtheorem{thm}{Theorem}[section]
\newtheorem{prop}[thm]{Proposition}
\newtheorem{cor}[thm]{Corollary}
\newtheorem{lem}[thm]{Lemma}

\theoremstyle{definition}
\newtheorem{dfn}[thm]{Definition}

\newtheorem{rmk}[thm]{Remark}

\def\dim{\mathop{\mathrm{dim}}\nolimits}

\def\Hom{\mathop{\mathrm{Hom}}\nolimits}

\def\<{{\langle}}
\def\>{{\rangle}}

\newcommand{\End}[1]{{\mathrm{End}({#1})}}

\def\1{\mathop{\mathbbm{1}}\nolimits}

\def\+{\mathop{\oplus}\nolimits}

\def\Supp{\mathop{\mathrm{Supp}}\nolimits}

\def\id{\mathop{\mathrm{id}}\nolimits}
\def\Spec{\mathop{\mathrm{Spec}}\nolimits}

\newcommand{\HNF}[3]{{
\xymatrix{
0	\ar[r]	&	{#2}_1	\ar[r]\ar[d]	&	{#2}_2	\ar[r]\ar[d]	&	\cdots\ar[r]	&	{#2}_{n-1}	\ar[r]\ar[d]	&	{#2}_n={#1}\ar[d]	\\
			&	{#3}_1\ar@{-->}[ul]	&	{#3}_2\ar@{-->}[ul] &					&	{#3}_{n-1}\ar@{-->}[ul]		&	{#3}_n\ar@{-->}[ul]	
}
}}

\newcommand{\ann}[1]{{\mathrm{ann}({#1})}}
\newcommand{\Stab}[1]{{\mathsf{Stab}\,{#1}}}

\newcommand{\spec}[1]{{\mathcal{S}\mathit{pec}\, {#1}}}

\newcommand{\mf}[1]{{\mathfrak{#1}}}

\newcommand{\bb}[1]{{\mathbb{#1}}}
\newcommand{\mca}[1]{{\mathcal{#1}}}
\newcommand{\mr}[1]{{\mathrm{#1}}}

\newcommand{\mb}[1]{{\mathbf{#1}}}

\title{$R$-linear triangulated categories and stability conditions}
\author{Kotaro Kawatani\\ Appendix by Hiroyuki Minamoto}
\email{kawatanikotaro@gmail.com}
\date{\today}


\begin{document}
\begin{abstract}
Let $R$ be a commutative ring. 
We introduce the notion of support of a object in an 
$R$-linear triangulated category. 
As an application, we study the non-existence of 
Bridgeland stability condition on $R$-linear triangulated categories. 
\end{abstract}

\maketitle	

\section{Introduction}

Recall that a stability condition introduced by 
Bridgeland \cite{MR2373143} on a triangulated category $\mb D$ 
defines a stability of objects in the category $\mb D$. 
If the category is the bounded derived category $\mb D^{b}(X)$
 of coherent sheaves on a projective variety $X$ (over a field), 
Analogously to the 
slope/Gieseker-Maruyama stability for coherent sheaves, 
one can define the stability of complexes of coherent sheaves via 
stability conditions on $\mb D^{b}(X)$ 
if it does exist. 
Recalling the stability of sheaves are 
defined by an ample line bundle on the variety, 
a stability condition on $\mb D^{b}(X)$ can be regarded as 
a huge generalization of ample line bundles on the projecitve 
variety. 
Thus the stability condition gives a powerful tool 
to study birational geometry for moduli spaces 
not only of sheaves, but also of complexes 
and much work has been done in this direction. 
For instance, moduli spaces on K3 surfaces or abelian sufraces 
are studied by many authors (Arcara and Bertram in \cite{MR2807848} or \cite{MR2998828}, 
Bayer and Macri in \cite{MR3279532} and \cite{MR3194493} or  
Minaminde, Yanagida and Yoshioka in \cite{MR3757471}.  
the case of other surfaces are studied by 
\cite{MR3010070}, \cite{MR3921322}, \cite{MR3551850}, 
or \cite{MR4125513}

However the existence of stability conditions a triangulated category 
 is sensitive, even if the category 
 is the derived category $\mb D^{b}(X)$ of 
a smooth projective variety $X$.    
For instance, if $\dim X\leq 2$, then 
$\Stab{\mb D^{b}(X)}$ does exists. 
When $\dim X=3$, the existence follows from generalized 
Bogomolov inequality proposed by Bayer-Macri-Toda \cite{MR3121850}.

Let us discuss the case where stability conditions do not exist. 
Since a stability condition on a triangulated category naturally 
induces a bounded $t$-structure, 
if the category has no bounded $t$-structure, 
then there are no stability condition. 
Such a category can be found by using schemes 
with singularities or varieties which are non-proper. 
For instance, as mentioned by Antieau, Gepner, 
and Heller \cite{MR3935042}, if the scheme $X$ is a nodal cubic curve, 
then the triangulated category $\mb D^{\mr{perf}}(X)$ of 
perfect complexes on $X$ has no bounded $t$-structure and 
hence no stability condition. 
They also conjecture that the category of 
perfect complexes on a finite dimensional Noetherian scheme 
$X$ has a bounded $t$-structure if and only if $X$ is regular. 
If the scheme $X$ is affine, 
the conjecture holds by Smith \cite{MR4383013}. 
Recently Neeman studies the conjecture for separated case 
(cf. \cite{https://doi.org/10.48550/arxiv.2202.08861}).

Let us introduce another example for the non-existence. 
Suppose that $X$ is an affine scheme $\Spec\, R$ 
of a Noetherian ring $R$. 
Then $\Stab{\mb D^{b}(X)}$ is non-empty if and only if 
$\dim X=0$ by the author \cite{kawatani2020stability}. 
Unlike the case of perfect complexes, 
the category $\mb D^{b}(X)$ has a natural bounded $t$-structure. 
In stead of the non-existence of $t$-structures, 
our proof is based on the supports of complexes in $\mb D^{b}(X)$. 

The aim of this paper is to extend the results in 
\cite{kawatani2020stability} to more general 
$R$-linear triangulated categories. 
The main theorem is the following: 

\begin{thm}\label{thm:kanban}
Let $R$ be a Noetherian ring and 
let $\mca X \to \Spec R$ be a proper morphism of schemes. 
If the dimension of the image is positive, 
then $\mb D^{b}(\mca X)$ and $\mb D^{\mr{perf}}(\mca X)$ 
have no stability condition. 
\end{thm}

Recall that the support 
$\Supp E$ of a complex $E$ of (quasi-)coherent sheaves is 
the union of the support of each cohomology $H^{i}(E)$. 
To develop the argument in \cite{kawatani2020stability}, 
we give a generalization of the support for objects 
in an $R$-linear triangulated category $\mb D$. 
Precisely if $E$ is an object in the $R$-linear category $\mb D$, 
we define $\Supp _{R} E$ as the support of the 
$R$-module $\Hom_{\mb D}(E, E)$ of endomorphisms. 
The generalized support $\Supp _{R}E$ coincides with $\Supp E$ 
if $E$ is a bounded complex of (not necessarily finite) $R$-modules.

Using the generalized support $\Supp _{R}E$, 
we give a sufficient condition for the non-existence of 
stability conditions under a certain finiteness assumption on $\mb D$. 
The properness in Theorem \ref{thm:kanban} guarantee the finiteness. 

For a familiy $\mca X \to S$ of schemes, 
our theorem says that there is no ``absolute'' stability conditions 
unless the dimension of the base is zero. 
On the other hand, 
a stability condition is one of generalizations of slope stability 
of locally free sheaves on projective varieties. 
Recall that the slope stability derived from an ample line bundle 
and that we have relatively ample line bundles for a family 
$\mca X \to S$ of projective varieties.  
Thus it might be natural to introduce a notion of ``relative'' stability 
conditions so that a stability condition over a base does exist. 
Interestingly the notion of stability conditions over a base scheme 
for a flat family of projective varieties is introduced by 
Bayer et al. in \cite{MR4292740} which gives a kind of ``relative'' 
stability conditions, to 
study Kuznetsov's non-commutative K3 surfaces. 
Our theorem gives an algebraic evidence of the necessity of 
relative stability conditions.

\section{$R$-linear triangulated categories}
From now on, $R$ is a commutative ring. 

\begin{dfn}
Let $\mb D$ be an $R$-linear triangulated category 
and $M$ an object in $\mb D$. 
\begin{enumerate}
\item We denote by $\mu \colon R \to \Hom_{\mb D}(M,M)$ 
the morphism defined by the following composition: 
\begin{equation}
  R \times \{ \id_{M}	\}
\subset
R \times \Hom_{\mb D}(M,M)
	\to \Hom_{\mb D}(M,M). 
\end{equation}
\item 
The support of the object $M$ is defined as the support of 
$\Hom_{\mb D}(M,M)$ as $R$-modules and is denoted by $\Supp _{R}M$: 
\[
\Supp _{R}M	:=	\Supp \Hom_{\mb D}(M,M). 
\]
\end{enumerate}
\end{dfn}

\begin{lem}\label{lem:sub-quotient-closed}
Let $\mb D$ be an $R$-linear triangulated category. 
Suppose that a distinguished triangle 
$A \overset{i}{\to} B \overset{p}{\to} C$ in 
$\mb D$ satisfies $\Hom_{\mb D}^{0}(A,C)=\Hom_{\mb D}^{-1}(A,C)=0$. 
\begin{enumerate}
\item Any morphism $\varphi \colon B \to B$ uniquely indues 
morphisms  
$\varphi_{A} \colon A \to A$ and 
$\varphi_{C} \colon C \to C$ such that 
$i \cdot \varphi _{A}= \varphi _{B} \cdot i$ and 
$p \cdot \varphi _{B} =\varphi _{C}\cdot p$. 
\item If $\varphi=0$ then the morphisms $\varphi_{C}$ and 
$\varphi_{A}=0$ are zero. 
\end{enumerate}
\end{lem}

\begin{proof}
We obtain the diagram of exact sequences of $R$-modules: 
\[
\xymatrix{
	&		&	\Hom_{\mb D}(A,A)	\ar[d]^{\alpha}	\\	
	&	\Hom_{\mb D}(B,B)\ar[r]^{i^{*}}\ar[d]_{p_{*}}	&	
			\Hom_{\mb D}(A,B)\ar[d]	\\
			\Hom_{\mb D}(C,C)\ar[r]	^{\gamma} &	
				\Hom_{\mb D}(B,C)\ar[r]	&	\Hom_{\mb D}(A,C)
}
\]
The assumptions imply both $\alpha$ and $\gamma$ are isomorphisms. 
Then $\varphi_{A}$ and $\varphi_{C}$ are given by 
\[
\varphi_{A}=\alpha^{-1}(i_{*} \varphi)
\text{ and }
\varphi_{C}= \gamma^{-1}(p_{*}\varphi). 
\]
The second assertion is obvious from the above diagram. 
\end{proof}

\begin{prop}
Let $\mb D$ be an $R$-linear triangulated category. 
Suppose that a distinguished triangle 
$A \overset{i}{\to} B \overset{p}{\to} C$ in $\mb D$ satisfies 
$\Hom_{\mb D}^{0}(A,C)=\Hom_{\mb D}^{-1}(A,C)=0$.

Then the following holds: 
\begin{equation}
\Supp _{R}A \cup \Supp _{R}C \subset \Supp _{R}B
\end{equation}
\end{prop}

\begin{proof}
Suppose that $\mf p \not \in \Supp_{R}B$. 
It is enough to show that $\mf p \not \in \Supp_{R}A$ and 
$\mf p \not \in \Supp_{R}C$. 
By the assumption $\mf p \not \in \Supp_{R}B$, 
there exists $r \in R -\mf p$ such that 
$r \cdot \id _{}=\mu_{r}=0$ in $\End{B}$. 
Since $\mb D$ is $R$-linear, 
the endomorphism $\mu_{r} \in \End{B}$ also induces 
the endomorphisms of $A$ and $C$ 
which make the following diagram commutative: 
\[
\xymatrix{
A	\ar[r]^{i}	\ar[d]_{\mu_{r}}	&	B	\ar[r]^{p}	\ar[d]_{\mu_{r}}	&	C	\ar[d]_{\mu_{r}}	\\
A	\ar[r]_{i}			&	B	\ar[r]_{p}			&	C	
}
\] 
By Lemma \ref{lem:sub-quotient-closed}, 
both $\mu_{r}$ and $\mu_{r}$ are zero. 
Thus 
the localizations $\End{A}\otimes _{R}R_{\mf p}$ and 
$\End{C}\otimes _{R}R_{\mf p}$ are zero. 
\end{proof}

\begin{cor}\label{cor:support-cohomology}
Let $R$ be an $R$-linear triangulated category. 
The $i$th cohomology of $E \in \mb D$ with respect to 
a bounded $t$-structure $(\mb D^{\leq 0}, \mb D^{\geq 1})$ on 
$\mb D$ is denoted by $H^{i}(E)$. 
Then the following holds: 
\[
\Supp_{R} H^{i}(E) \subset \Supp _{R} E . 
\]
\end{cor}

\begin{proof}
Set $p$ and $q$ by 
\begin{align*}
p	&= \max\{	i \in \bb Z \mid H^{i}(E) \neq 0	\}, \text{ and }	\\
q	&=\min\{		i \in \bb Z \mid H^{i}(E) \neq 0	\}. 
\end{align*}
The proof is by induction on $p-q$. 
If $p-q=0$, then the assertion is clear. 

Taking the filtration with respect to the $t$-structure, 
we obtain the following triangle 
\[
\xymatrix{
E^{p-1}	\ar[r]	&	E	\ar[r]	&	H^{p}(E)[-p]. 
}
\]
where $E ^{p-1} \in \mb D^{\leq 0}[1-p]$. 
Lemma \ref{lem:sub-quotient-closed} implies 
$\Supp _{R}E^{p-1}\subset \Supp _{R}E$ and 
$\Supp_{R}H^{p}(E) \subset \Supp _{R}E$. 
Then the assumption of induction implies the desired assertion. 
\end{proof}

\begin{lem}\label{lem:extension-closed}
Let $\mb D$ be an $R$-linear triangulated category. 
Consider a diagram of distinguished triangle in $\mb D$: 
\[
\xymatrix{
A \ar[r]^{i}\ar[d]_-{\psi_{A}}	&	B	\ar[r]^{p}\ar[d]_-{	\psi_{B}}&	C	\ar[d]_-{\psi_{C}}	\\
A	\ar[r]^{i}		&	B	\ar[r]^{p}		&	C		.
}
\]
If 
$\psi _{A}=0$ and $\psi _{C}=0$ 
then the composite $\psi _{B}^{2}$ is zero. 
\end{lem}

\begin{proof}
Since $p\cdot \psi _{B} = \psi _{C}\cdot p=0$, there exists a 
$\varphi \colon B \to A$ such that 
$i \cdot \varphi =\psi _{B}$. 
Thus we see \[
\psi _{B}^{2} = \psi _{B} \cdot  i \cdot  \varphi = i \cdot \psi _{A} \cdot \varphi =0. 
\]
\end{proof}

\begin{lem}\label{lem:BinAorC}
Let 
$A \overset{i}{\to} B \overset{p}{\to}C$
be a distinguished triangle in an $R$-linear triangulated 
category $\mb D$. 
Then the following holds: 
\[
\Supp _{R}B \subset \Supp _{R}A \cup \Supp _{R}C. 
\]
\end{lem}

\begin{proof}
Take a prime ideal $\mf p$ such that $\mf p \not \in \Supp_{R}A$ and 
$\mf p \not \in \Supp_{R}C$. 
Then there exists $r_{A} $ (resp. $r_{C}$) in $R -\mf p$ such that 
$r_{A} \cdot \id _{A}=0$ (resp. $r_{C} \cdot \id _{C}=0$). 
Then $r_{0}=r_{A}r_{C}$ satisfies $r_{0}\id_{A}=0$ and 
$r_{0}\id_{C}=0$. 
Since the category $\mb D$ is $R$-linear, 
the following diagram commutes: 
\[
\xymatrix{
A	\ar[r]^{i}	\ar[d]_{\mu_{r_{0}}}	&	B	\ar[r]^{p}	\ar[d]_{\mu_{r_{0}}}	&	C	\ar[d]_{\mu_{r_{0}}}	\\
A	\ar[r]_{i}				&	B	\ar[r]_{p}			&	C. 
}
\]
By Lemma \ref{lem:extension-closed}, 
we see $\mu_{r_{0} }^{2}= \mu_{r_{0}^{2}} \colon B \to B$ 
is the zero morphism. 
Hence $\id _{B} \in \End{B}$ is zero via localization to $\mf p \in \Spec R$. 
Thus we see $\mf p \not \in \Supp _{R}B= \Supp{\End{B}}$. 
\end{proof}

\begin{cor}\label{cor:supports-cohomology}
Let $\mb D$ be an $R$-linear triangulated category. 
The $i$th cohomology of $E \in \mb D$ with respect to a 
bounded $t$-structure on $\mb D$ is denoted by $H^{i}(E)$. 
The the following holds: 
\[
\Supp_{R}E	=	\bigcup_{i \in \bb Z}	\Supp_{R}H^{i}(E). 
\]
\end{cor}

\begin{proof}
By Corollary \ref{cor:support-cohomology}, it is enough to show 
$\Supp _{R}E \subset \bigcup_{i \in \bb Z}	\Supp_{R}H^{i}(E)$. 
Set $p$ and $q$ by 
\begin{align*}
p	&= \max\{	i \in \bb Z \mid H^{i}(E) \neq 0	\}, \text{ and }	\\
q	&=\min\{		i \in \bb Z \mid H^{i}(E) \neq 0	\}. 
\end{align*}
The proof is by the induction on $p-q$. 
Similarly to the proof of Corollary \ref{cor:support-cohomology}, 
we have the distinguished triangle 
\[
\xymatrix{
E^{p-1}	\ar[r]	&	E	\ar[r]	&	H^{p}(E)[-p]. 
}
\]
Lemma \ref{lem:BinAorC} implies 
\[
\Supp _{R}E  \subset \Supp _{R} E^{p-1}  \cup \Supp_{R} H^{p}(E). 
\]
Then the assumption of the induction 
implies 
\[
\Supp _{R}E^{p-1} =\bigcup_{i\in \bb Z}H^{i}(E^{p-1})
\]
which completes the proof. 
\end{proof}

In the last of this section, 
we show that the generalized support coincides with 
``usual supports'' of complexes of $R$-modules. 
So let us suppose a triangulated category is 
the unbounded derived category $\mb D(\mr{Mod}\, R)$ of 
(not necessarily finite) $R$-modules. 
Recall that the support of an object $E$ in $\mb D(\mr{Mod}\, R)$ 
is the union of the supports of the $i$th cohomology $H^{i}(E)$: 
\[
\Supp E := \bigcup_{i\in \bb Z} \Supp H^{i}(E). 
\]

\begin{lem}\label{lem:subset-affine}
Let $R$ be commutative ring and let $E$ be a complex of $R$-modules. 

\begin{enumerate}
\item We have $\Supp E \subset \Supp _{R}E$. 
\item If the complex $E$ is bounded, then $\Supp E = \Supp _{R}E$. 
\end{enumerate}
\end{lem}

\begin{proof}
Suppose $\mf p \not\in \Supp _{R}E$. 
Then there exists $r \in R -\mf p$ such that 
$r \id = \mu_{r}$ is zero in $\Hom_{}(E, E)$. 
Taking cohomology with respect to the standard $t$-structure, 
$\mu_{r}$ does imply the multiplication $\mu_{r}$ on $H^{i}(E)$. 
Since $\mu_{r} \in \End E$ is zero, 
so does $\mu_{r} \in \End {H^{i}(E)}$ for any $i \in \bb Z$. 
Hence $\mf p $ is not in $\Supp H^{i}(E)$ for any $i \in \bb Z$. 
Thus we have 
\[
\Supp E =\bigcup _{i \in \bb Z} \Supp H^{i}(E)  \subset \Supp _{R}E. 
\]

Suppose that $E$ is bounded and 
take a $\mf p \in \Spec\, R - \bigcup _{i \in \bb Z} \Supp H^{i}(E)$. 
Then 
there exists $ r_{i} \in R - \mf p$ such that 
$\mu_{r_{i}} \colon H^{i}(E) \to H^{i}(E)$ is zero 
for each $i$ with $H^{i}(E)\neq 0$. 
Since $E$ is bounded, let $s$ be the product of such $r_{i}$.

Then the morphism $\mu_{s} \in \End{H^{i}(E)}$ is zero 
for any $i \in \bb Z$.  
By the lemma below, 
there exists $N \in \bb N$ such that 
the $N$th composite $\mu_{s}^{N}=\mu_{s^{N}}$ is 
the zero morphism of $E$. 
Hence we see $\mf p \not \in \Supp _{R}E$. 
\end{proof}

\begin{lem}
Let $E$ be a bounded complex of $R$-modules. 
For an $r\in R$, if $\mu_{r} \in \End{H^{i}(E)}$ is zero 
for any $i \in \bb Z$, 
then there exists $N \in \bb N$ such that 
the composite $\mu_{s}^{N}=\mu_{s^{N}}$ is zero in $\End{E}$. 
\end{lem}

\begin{proof}
Set 
$p:= \max \{	i \in \bb Z \mid H^{i}(E)\neq 0	\}$ and 
$q :=\{	i \in \bb Z	\mid H^{i}(E)\neq 0 \}$. 
The proof is by induction on $p-q$. 

If $p-q=0$, then $N$ is taken to be $1$. 
Suppose the assertion holds for $p-q=\ell-1$. 
Taking truncation of $E$, 
we have the following distinguished triangle: 
\[
\xymatrix{
E^{p-1}	\ar[r]	&	E	\ar[r]	&	H^{p}(E)[-p]. 
}
\]
By the assumption on induction, there exists 
$N \in \bb N$ such that $\mu_{s}^{N} \colon E^{p-1} \to E^{p-1}$ is zero. 
Since $\mu_{s}^{N} \colon H^{p}(E)\to H^{p}(E)$ is zero, 
Lemma \ref{lem:extension-closed} implies that  
the endomorphism $\mu_{s}^{2N} \colon E \to E$ is zero. 
\end{proof}

\section{Stability conditions and $R$-linear categories}

The aim in this section is to study a property of 
$\Supp _{R}E$ for $\sigma$-stable objects with 
respect to a stability condition $\sigma$. 
The following definition reflects properties of stable objects. 

\begin{dfn}
Let $\mb D$ be an $R$-linear triangulated category. 
An object $M \in \mb D$ satisfies the isomorphic property 
if the following holds: 
\begin{itemize}
\item[(Ism)] $\forall r \in R$, $\mu_{r} \colon M \to M$ is an 
isomorphism if $\mu_{r}$ is non-zero. 
\end{itemize}
\end{dfn}

For well behavior of $\Supp _{R}E$, we need a finiteness assumption on the triangulated category $\mb D$. 
\begin{dfn}
Let $\mb D$ be an $R$-linear triangulated category. 
$\mb D$ is said to be \textit{finite} over $R$ if 
the $R$-module $\Hom_{\mb D}(E, E)$ is finite for any object $E\in \mb D$. 
\end{dfn}

\begin{lem}\label{lem:vanishing-domain}
Assume that the commutative ring $R$ is an integral domain 
with $\dim R>0$ and an $R$-linear triangulated category 
$\mb D$ is finite. 
If an object $M$ in $\mb D$ satisfies 
\begin{itemize}
\item the morphism $\mu_{r} \colon M \to M$ is an 
isomorphism for any $r \in R \setminus\{	0	\}$, 
\end{itemize}
then $M$ is zero. 
\end{lem}

\begin{proof}
It is enough to show $\Hom_{\mb D}(M,M)=0$, 
since the category $\mb D$ is additive. 

If $\Hom_{\mb D}(M,M)$ is non-zero, then we have $\Supp_{R} M=\Spec R$ 
since $\Hom_{\mb D}(M,M)$ contains $R$. 
On the other hand, there exists a non-unit $r$ in $R$. 
Then the assumption implies that $\mu_{r}$ gives an isomorphism of $M$. 
Hence we have the following isomorphism 
\[
\mu_{r}^{*} \colon \Hom_{\mb D}(M,M) \to \Hom_{\mb D}(M,M). 
\]
Hence the $R$-module $\Hom_{\mb D}(M,M)$ satisfies 
$\Hom_{\mb D}(M,M) \otimes R/(r)=0$ which implies 
$\Supp _{R}M \cap \Spec R/(r) = \emptyset$ since 
$\Hom_{\mb D}(M,M)$ is finitely generated. 
This contradicts the fact $\Supp_{R}M=\Spec R$. 
\end{proof}

\begin{lem}\label{lem:dim=0}
Suppose an object $M$ in a finite $R$-linear triangulated category 
$\mb D$ satisfies the condition (Ism). 
Then the following holds: 
\begin{enumerate}
\item The ideal $\ann{\End{M}}$ of $R$ is prime. 
\item If $M$ is non-zero, then $\ann{\End{M}} = \ann{f}$ 
for any $f \in \End{M} \setminus \{ 0 \}$. 
\item If $M$ is non-zero, then $\dim \Supp_{R}(M)=0$. 
\end{enumerate}
\end{lem}

\begin{proof}
Let $a$ and $b$ be in $R$. 
Suppose the product $ab$ is in $\ann {\End{M}}$ and 
$a \not \in \ann{\End{M}}$. 
Then the endomorphism $\mu_{a}$ is non-zero and hence is invertible. 
Thus $\mu_{b}$ is zero by $\mu_{ab}=\mu_{a}\mu_{b}$. 
Hence $b \in \ann{\End{M}}$.

Take $a \in \ann{f} $. 
If $\mu_{a}$ is non-zero, 
then $\mu_{a}$ is an isomorphism by the assumption. 
This contradicts $a \in \ann{f}$. 
Thus $\mu_{a}$ is the zero-morphism and 
we have $\ann{f} \subset \ann{\End{M}}$. 
The opposite inclusion is clear.

Put $\mf p=\ann{\End{M}}$. 
$\End{M}$ is an $R /\mf p$-module and 
$M$ satisfies the condition (Ism). 
If the prime ideal $\mf p$ is not maximal, 
then $\dim R/\mf p>0$ and Lemma \ref{lem:vanishing-domain} implies $M=0$. 
Hence $\mf p$ has to be maximal and $\dim \Supp_{R}M=0$. 
\end{proof}

\begin{lem}\label{lem:stable-dim0}
Let $\mb D$ be a finite $R$-linear triangulated category and 
let $\sigma$ be a locally finite stability condition on $\mb D$. 
If an object $E \in \mb D$ is $\sigma$-stable then $\dim \Supp_{R}E=0$. 
\end{lem}

\begin{proof}
If $E$ is $\sigma$-stable, then the condition (Ism) holds. 
Lemma \ref{lem:dim=0} implies $\dim \Supp_{R}E=0$. 
\end{proof}

\begin{prop}\label{prop:semistable-dim0}
Let $\mb D$ be a finite $R$-linear triangulated category and 
$\sigma$ be a locally finite stability condition on $\mb D$. 
If an object $E \in \mb D$ is $\sigma$-semistable, then 
$\dim \Supp _{R}E=0$. 
\end{prop}

\begin{proof}
Let $\mca P(\phi)$ be the slicing of $\sigma$ with phase $\phi$. 
Recall that $\mca P(\phi)$ is an abelian category. 
Since $\sigma$ is locally finite, 
any object $E \in \mca P(\phi)$ is given by  
a successive extension of finite $\sigma$-stable objects. 

The proof is induction  on the number of 
stable factors of $E \in \mca P(\phi)$. 
If the number is $1$, 
the assertion follows from Lemma \ref{lem:stable-dim0} 
since $E$ is $\sigma$-stable.

Now $E$ is not $\sigma$-stable but $\sigma $-semistable. 
Take a $\sigma$-stable subobject $A$ of $E$. 
Then the quotient $E/A$ satisfies the assumption on the induction. 
Hence Lemma \ref{lem:BinAorC} implies 
\[
\Supp_{R}E \subset
\Supp_{R}A \cup \Supp_{R}E/A. 
\]
and the assumption on induction implies $\dim \Supp_{R}E/A=0$. 
Then Corollary \ref{cor:supports-cohomology} 
implies the desired assertion. 
\end{proof}

\begin{thm}\label{thm:support-property}
Let $\mb D$ be a finite $R$-linear triangulated category. 
Suppose that there exists a locally finite 
stability condition $\sigma$ on $\mb D$. 
For any object $E \in \mb D$, 
the dimension of the support $\Supp _{R}E$ is zero. 
Moreover any $i$th cohomology $H^{i}(E)$ of 
$E$ has the zero dimensional support.  
\end{thm}

\begin{proof}
The proof is by induction on the length 
$\ell (E)$ of the Harder-Narasimhan filtration of 
$E$ with respect to $\sigma$. 
If $\ell (E)=1$, the assertion follows from 
Proposition \ref{prop:semistable-dim0}. 

Suppose that the assertion holds for $\ell(E) -1$.  
Taking the last term of the Harder-Narasimhan filtration of $E$, 
we obtain the distinguished triangle 
\begin{equation}\label{eq:filtration}
\xymatrix{
E_{n-1}	\ar[r]	&	E	\ar[r]	&	A_{n}, 
}
\end{equation}
where $A_{n}$ is $\sigma$-semistable and $\ell (E_{n-1})=\ell (E)-1$. 
Thus we see $\dim \Supp_{R}E_{n-1}=0$ and $\dim \Supp _{R}A_{n}=0$. 
Then Lemma \ref{lem:BinAorC} implies the desired assertion. 
The last assertion follows from Corollary \ref{cor:supports-cohomology}. 
\end{proof}

\begin{cor}\label{cor:matome}
Let $\mb D$ be a finite $R$-linear triangulated category. 
If there exists an object $M \in \mb D$ such that $\dim \Supp_{R} M>0$, 
then there exists no locally finite stability condition on $\mb D$. 
\end{cor}

\begin{proof}
If there exists a locally finite stability condition on $\mb D$, 
the dimension the support of any object in $\mb D$ is zero. 
This contradicts Thoerem \ref{thm:support-property}. 
\end{proof}

\begin{cor}\label{cor:generalization}
	Let $f \colon \mca X \to \Spec R$ be a proper morphism to the 
	affine Noetherian scheme $\Spec R$. 
	If the dimension of the image of $f$ is positive, then 
	both $\Stab{\mb D^{b}(\mca X)}$ and 
	$\Stab{\mb D^{\mr{perf}}(\mca X)}$ are empty. 
\end{cor}

\begin{proof}
Recall that $\mb D^{\mr {perf}}(\mca X)$ is 
a full subcategory of $\mb D^{b}(\mca X)$  
and the structure sheaf $\mca O_{\mca X}$ is in 
$\mb D^{\mr{perf}}(\mca X)$. 

Let $Z$ be the image of $f$. 
 Note that $Z$ is a closed subscheme of $\Spec R$.  
 Then the morphism $\tilde f \colon \mca X \to Z$ is proper and hence 
 $\mb D^{b}(X)$ is linear and finite over the ring $H^{0}(Z, \mca O_{Z})$.
 Since $\tilde f$ is surjective, 
 the ring $H^{0}(X, \mca O_X)$ containes $H^0(Z, \mca O_{Z})$. 
 Then the assumption implies 
\[
\dim \Supp _{H^{0}(Z, \mca O_{Z})} \mca O_{X}
\geq \dim \Supp _{H^{0}(Z, \mca O_{Z})}H^{0}(Z, \mca O_{Z})>0. 
\]  
Then Corollary \ref{cor:matome} implies the desired assertion. 
\end{proof}

\begin{rmk}
In \cite{kawatani2020stability}, we show that 
the bounded derived category of a affine Noetherian 
scheme $\Spec R$ has a locally finite stability condition. 
Corollary \ref{cor:generalization} gives a generalization. 
\end{rmk}

\subsection*{Acknowledgment} 
The author is partially supported by JSPS KAKENHI Grant Number JP21K03212. 
He also thanks H.Minamoto for valuable comments for this article. 

\subsection*{Conflict of Interest}
Not applicable.

\subsection*{Availability of Data and Materials}
Not applicable.

\appendix
\section{by Hiroyuki Minamoto}

Inspired by main body this paper, we prove the following theorem.

\begin{thm}\label{thm:minamoto2}
Let $X$ be a Noetherian scheme. 
Assume that the Krull dimension of the ring 
$\Gamma (X, \mca O_X)$ is positive. 
Then both $\Stab{\mb D^{b}(X)}$ and 
	$\Stab{\mb D^{\mr{perf}}( X)}$ are empty. 
\end{thm}

We need preparations.

\begin{lem}\label{lem:minamoto}
Let $X$ be a Noetherian scheme, $s \in \Gamma(X, \mca O_X)$ a 
global section, 
and $Z:= \spec {\mca O_X/(s)}$ the vanishing locus of the section $s$. 
For an object $M \in \mb D^b(X)$, 
we denote by $\mu_s$ the multiplication map $\mu_{s}^M \colon M \to M$. 
Then for $M \in \mb D^b(X)$, the following assertion hold: 
\begin{enumerate}
  \item If $\mu_s$ is zero, then $\Supp M \subset Z$. 
  \item If $\mu_s$ is an isomorphism, then $\Supp M \subset X\setminus Z$. 
\end{enumerate}
\end{lem}

\begin{proof}
(1) By the assumption, 
the support of $i$th cohomology $H^i(M)$ of $M$ with respect to 
standard $t$-structure is contained in $Z$ since 
the cohomology of the morphism $\mu_s$ is zero for all $i \in \bb Z$. 
This gives the proof. 

(2) By the assumption, 
the multiplication map $\mu_s^i \colon H^i(M) \to H^i(M)$ is an 
isomorphism for each $i \in \bb Z$. 
Hence 
the localization $(\mu_s^i)_x \colon H^i(M)_x \to H^i(M)_x$ of the map 
$\mu_s^i$ on each $x \in X$ is an isomorphism of 
$\mca O_{X,x}$ modules. 
If $x$ in $Z$, then clearly the germ $s_x$ is in the maximal ideal 
$\mf{m}_{X,x}$ of $\mca O_{X,x}$. 
Hence we see $\mf m_{X,x} H^i(M)_x = H^i(M)_{x}$. 
By Nakayama's lemma, we have $H^i(M)_x$ is zero and hence 
$\Supp M \subset X \setminus Z$. 
\end{proof}

\begin{prop}\label{thm:minamoto}
  Let $X$ be a connected Noetherian scheme, 
  and $Z$ the vanishing locus $\spec {\mca O_X/(s)}$ 
  of a global section $s\in \Gamma (X, \mca O_X)$. 
  Assume that $s$ is neither nilpotent nor 
  invertible. 
  Then both $\Stab{\mb D^{b}( X)}$ and 
	$\Stab{\mb D^{\mr{perf}}( X)}$ are empty. 
\end{prop}

\begin{proof}
We deal with $\mb D^b(X)$. The same proof works for $\mb D^{\mr{perf}}( X)$.

Suppose to the contrary that 
$\mb D^b(X)$ has a locally finite 
stability condition. 
Then by the argument of the proofs of Proposition \ref{prop:semistable-dim0} 
and Theorem \ref{thm:support-property},  
  there exists finite objects 
$M_1, \cdots, M_n \in \mb D^b(X)$ having the property 
(Ism) such that 
$\mca O_X$ is in the thick hull 
of $M_1, \cdots, M_n$. 

We set 
\[
I = \{ i \mid 1 \leq i \leq n , \mu_s ^{M_i} \mbox{ is an isomorphism}\}, 
J=\{ j \mid 1 \leq j \leq n, \mu_s^{M_j} \mbox{is zero}\} . 
\]
Note that $I \sqcup J=\{1, 2, \cdots, n\}$. 
We set $Y_{I}:= \bigcup_{i\in I} \Supp M_i$ and 
$Y_{J} := \bigcup_{j \in J} \Supp M_j$. 
Then we see $Y_I \subset X \setminus Z$ and $Y_J \subset Z$ 
by Lemma \ref{lem:minamoto}. 
On the other hand, we have 
$X = \Supp \mca O_X \subset \bigcup _{i=1}^n \Supp M_i = Y_I \sqcup Y_J\subset Y_{I}\sqcup Z$ 
and hence $X = Y_I \sqcup Z$. 
Since $X$ is connected by the assumption, 
we have either $Z=X$ or $Z=\emptyset$. 
However the condition  $Z =X$ contradicts to the assumption $s$ is non-nilpotent 
and  the condition $Z= \emptyset$ contradicts to the assumption that $s$ is non-invertible. 
\end{proof}

We proceed a proof of the main theorem of appendix.

\begin{proof}[Proof of Theorem \ref{thm:minamoto2}]
Take a connected component $X'$ of $X$ such that 
$\dim \Gamma (X', \mca O_{X'}) \geq 1$. 
Then it has a global section which is neither nilpotent nor invertible. 
Then the assertion follows from Proposition \ref{thm:minamoto}. 
\end{proof}

\newcommand{\etalchar}[1]{$^{#1}$}


\begin{thebibliography}{BLM{\etalchar{+}}21}

\bibitem[AB11]{MR2807848}
Daniele Arcara and Aaron Bertram.
\newblock Reider's theorem and {T}haddeus pairs revisited.
\newblock In {\em Grassmannians, moduli spaces and vector bundles}, volume~14
  of {\em Clay Math. Proc.}, pages 51--68. Amer. Math. Soc., Providence, RI,
  2011.

\bibitem[AB13]{MR2998828}
Daniele Arcara and Aaron Bertram.
\newblock Bridgeland-stable moduli spaces for {$K$}-trivial surfaces.
\newblock {\em J. Eur. Math. Soc. (JEMS)}, 15(1):1--38, 2013.
\newblock With an appendix by Max Lieblich.

\bibitem[ABCH13]{MR3010070}
Daniele Arcara, Aaron Bertram, Izzet Coskun, and Jack Huizenga.
\newblock The minimal model program for the {H}ilbert scheme of points on
  {$\Bbb{P}^2$} and {B}ridgeland stability.
\newblock {\em Adv. Math.}, 235:580--626, 2013.

\bibitem[AGH19]{MR3935042}
Benjamin Antieau, David Gepner, and Jeremiah Heller.
\newblock {$K$}-theoretic obstructions to bounded {$t$}-structures.
\newblock {\em Invent. Math.}, 216(1):241--300, 2019.

\bibitem[BLM{\etalchar{+}}21]{MR4292740}
Arend Bayer, Mart\'{\i} Lahoz, Emanuele Macr\`\i, Howard Nuer, Alexander Perry,
  and Paolo Stellari.
\newblock Stability conditions in families.
\newblock {\em Publ. Math. Inst. Hautes \'{E}tudes Sci.}, 133:157--325, 2021.

\bibitem[BM14a]{MR3279532}
Arend Bayer and Emanuele Macr\`\i.
\newblock M{MP} for moduli of sheaves on {K}3s via wall-crossing: nef and
  movable cones, {L}agrangian fibrations.
\newblock {\em Invent. Math.}, 198(3):505--590, 2014.

\bibitem[BM14b]{MR3194493}
Arend Bayer and Emanuele Macr\`\i.
\newblock Projectivity and birational geometry of {B}ridgeland moduli spaces.
\newblock {\em J. Amer. Math. Soc.}, 27(3):707--752, 2014.

\bibitem[BMT14]{MR3121850}
Arend Bayer, Emanuele Macr\`\i, and Yukinobu Toda.
\newblock Bridgeland stability conditions on threefolds {I}:
  {B}ogomolov-{G}ieseker type inequalities.
\newblock {\em J. Algebraic Geom.}, 23(1):117--163, 2014.

\bibitem[Bri07]{MR2373143}
Tom Bridgeland.
\newblock Stability conditions on triangulated categories.
\newblock {\em Ann. of Math. (2)}, 166(2):317--345, 2007.

\bibitem[Kaw20]{kawatani2020stability}
Kotaro Kawatani.
\newblock Stability conditions on affine noetherian schemes, 2020.

\bibitem[LZ19]{MR3921322}
Chunyi Li and Xiaolei Zhao.
\newblock Birational models of moduli spaces of coherent sheaves on the
  projective plane.
\newblock {\em Geom. Topol.}, 23(1):347--426, 2019.

\bibitem[MYY18]{MR3757471}
Hiroki Minamide, Shintarou Yanagida, and Kota Yoshioka.
\newblock The wall-crossing behavior for {B}ridgeland's stability conditions on
  abelian and {K}3 surfaces.
\newblock {\em J. Reine Angew. Math.}, 735:1--107, 2018.

\bibitem[Nee22]{https://doi.org/10.48550/arxiv.2202.08861}
Amnon Neeman.
\newblock Bounded t-structures on the category of perfect complexes, 2022.

\bibitem[Nue16]{MR3551850}
Howard Nuer.
\newblock Projectivity and birational geometry of {B}ridgeland moduli spaces on
  an {E}nriques surface.
\newblock {\em Proc. Lond. Math. Soc. (3)}, 113(3):345--386, 2016.

\bibitem[NY20]{MR4125513}
Howard Nuer and Kota Yoshioka.
\newblock M{MP} via wall-crossing for moduli spaces of stable sheaves on an
  {E}nriques surface.
\newblock {\em Adv. Math.}, 372:107283, 119, 2020.

\bibitem[Smi22]{MR4383013}
Harry Smith.
\newblock Bounded t-structures on the category of perfect complexes over a
  {N}oetherian ring of finite {K}rull dimension.
\newblock {\em Adv. Math.}, 399:Paper No. 108241, 2022.

\end{thebibliography}

\end{document}